\documentclass[12pt,reqno]{article}

\usepackage{geometry}
\usepackage{amssymb}
\usepackage{amsmath}
\usepackage[shortlabels]{enumitem}
\usepackage{csquotes}

\usepackage{amsthm}

\theoremstyle{plain}
\newtheorem{theorem}{Theorem}
\newtheorem{corollary}[theorem]{Corollary}
\newtheorem{lemma}[theorem]{Lemma}
\newtheorem{proposition}[theorem]{Proposition}

\theoremstyle{definition}
\newtheorem{definition}[theorem]{Definition}
\newtheorem{example}[theorem]{Example}

\theoremstyle{remark}
\newtheorem{remark}[theorem]{Remark}

\usepackage[colorlinks=true,
urlcolor=blue,
linkcolor=blue,
filecolor=webbrown,
citecolor=red]{hyperref}

\usepackage{ tikz }
\usepackage{float}

\newcommand{\wei}{w}
\newcommand{\len}{\ell}

\begin{document}

\begin{center}
\vskip 1cm{\Large \bf 
Enumeration of S-omino towers and\\\rule{0pt}{2.5ex}
row-convex k-omino towers
}
\vskip 1cm
\large
Alexander M. Haupt\\
Institute of Mathematics\\
Hamburg University of Technology\\
Am Schwarzenberg-Campus 3\\
21073 Hamburg\\
Germany\\
\href{mailto:alexander.haupt@tuhh.de}{\tt alexander.haupt@tuhh.de}
\end{center}

\vskip .2 in

\begin{abstract}
	We first enumerate a generalization of domino towers that was proposed by Tricia M.\ Brown (J.\ Integer Seq.\ 20 (2017)), which we call S-omino towers. We establish equations that the generating function must satisfy and then apply the Lagrange inversion formula to find a closed formula for the number of towers. We also show a connection to generalized Dyck paths and provide an explicit bijection. Finally, we consider the set of row-convex k-omino towers, introduced by Brown, and calculate an exact generating function.
\end{abstract}

\section{Introduction}\label{domino:introduction}

In this paper we enumerate a generalization of so called \emph{domino towers}. Domino towers are two-dimensional structures made out of $n$ dominoes, i.e., rectangular blocks of width $2$ and height $1$, in a brickwork pattern, such that:
\begin{enumerate}[noitemsep]
	\item The dominoes on the bottom level are contiguous, i.e., the row is convex;
	\item Every domino above the bottom row is (half) supported on at least one domino in the row below it;
	\item No domino lies directly on top of another domino, such as in a brickwork pattern.
\end{enumerate}

See Figure~\ref{fig:domino:smallTowers} for the domino towers with $n \in \{1,2,3\}$. The problem of counting domino towers was first mentioned by Viennot \cite{MR768962} and subsequently by many others, e.g., see Zeilberger \cite{1208.2258} and Mikl\'os B\'ona \cite[p.\ 25]{miklosBonaHandbook}. Surprisingly, the number of domino towers made up of exactly $n$ blocks is simply $3^{n-1}$. There are proofs using generating functions and bijective proofs.

\begin{figure}
	\centering
	\begin{tikzpicture}[scale=0.35]
	\begin{scope}[shift={(0,1)}]
	\draw[fill=cyan!50, draw=black] (0,0) rectangle ++(2,1);
	\end{scope}
	
	\begin{scope}[shift={(6,1)}]
	\draw[fill=cyan!50, draw=black] (0,0) rectangle ++(2,1);
	\draw[fill=cyan!50, draw=black] (2,0) rectangle ++(2,1);
	\end{scope}
	
	\begin{scope}[shift={(12,1)}]
	\draw[fill=cyan!50, draw=black] (0,0) rectangle ++(2,1);
	\draw[fill=cyan!50, draw=black] (1,1) rectangle ++(2,1);
	\end{scope}
	
	\begin{scope}[shift={(18,1)}]
	\draw[fill=cyan!50, draw=black] (1,0) rectangle ++(2,1);
	\draw[fill=cyan!50, draw=black] (0,1) rectangle ++(2,1);
	\end{scope}
	
	\begin{scope}[shift={(0,-3)}]
	\draw[fill=cyan!50, draw=black] (0,0) rectangle ++(2,1);
	\draw[fill=cyan!50, draw=black] (2,0) rectangle ++(2,1);
	\draw[fill=cyan!50, draw=black] (4,0) rectangle ++(2,1);
	\end{scope}
	
	\begin{scope}[shift={(9,-3)}]
	\draw[fill=cyan!50, draw=black] (0,0) rectangle ++(2,1);
	\draw[fill=cyan!50, draw=black] (2,0) rectangle ++(2,1);
	\draw[fill=cyan!50, draw=black] (3,1) rectangle ++(2,1);
	\end{scope}
	
	\begin{scope}[shift={(18,-3)}]
	\draw[fill=cyan!50, draw=black] (0,0) rectangle ++(2,1);
	\draw[fill=cyan!50, draw=black] (2,0) rectangle ++(2,1);
	\draw[fill=cyan!50, draw=black] (1,1) rectangle ++(2,1);
	\end{scope}
	
	\begin{scope}[shift={(0,-7)}]
	\draw[fill=cyan!50, draw=black] (0,0) rectangle ++(2,1);
	\draw[fill=cyan!50, draw=black] (2,0) rectangle ++(2,1);
	\draw[fill=cyan!50, draw=black] (-1,1) rectangle ++(2,1);
	\end{scope}
	
	\begin{scope}[shift={(9,-7)}]
	\draw[fill=cyan!50, draw=black] (0,0) rectangle ++(2,1);
	\draw[fill=cyan!50, draw=black] (1,1) rectangle ++(2,1);
	\draw[fill=cyan!50, draw=black] (-1,1) rectangle ++(2,1);
	\end{scope}
	
	\begin{scope}[shift={(18,-7)}]
	\draw[fill=cyan!50, draw=black] (0,0) rectangle ++(2,1);
	\draw[fill=cyan!50, draw=black] (1,1) rectangle ++(2,1);
	\draw[fill=cyan!50, draw=black] (2,2) rectangle ++(2,1);
	\end{scope}
	
	\begin{scope}[shift={(27,1)}]
	\draw[fill=cyan!50, draw=black] (0,0) rectangle ++(2,1);
	\draw[fill=cyan!50, draw=black] (1,1) rectangle ++(2,1);
	\draw[fill=cyan!50, draw=black] (0,2) rectangle ++(2,1);
	\end{scope}
	
	\begin{scope}[shift={(27,-3)}]
	\draw[fill=cyan!50, draw=black] (0,0) rectangle ++(2,1);
	\draw[fill=cyan!50, draw=black] (-1,1) rectangle ++(2,1);
	\draw[fill=cyan!50, draw=black] (0,2) rectangle ++(2,1);
	\end{scope}
	
	\begin{scope}[shift={(27,-7)}]
	\draw[fill=cyan!50, draw=black] (0,0) rectangle ++(2,1);
	\draw[fill=cyan!50, draw=black] (-1,1) rectangle ++(2,1);
	\draw[fill=cyan!50, draw=black] (-2,2) rectangle ++(2,1);
	\end{scope}
	
	\end{tikzpicture}
	\caption{Small restricted domino towers for $n \in \{1,2,3\}$.}
	\label{fig:domino:smallTowers}
\end{figure}

In this paper we drop the restriction that blocks cannot be placed directly on top of another and call the structures \textit{unrestricted} towers. In Corollary~\ref{domino:4^n} we show that the number of unrestricted domino towers is $4^{n-1}$. See Figures~\ref{fig:domino:dominoTowerRestricted} and~\ref{fig:domino:dominoTowerUnrestricted} for examples of restricted and unrestricted domino towers. 
\begin{figure}
	\centering
	\begin{minipage}{.42\textwidth}
		\centering
		\begin{tikzpicture}[scale=0.5]
		\draw[fill=cyan!50, draw=black] (0,0) rectangle ++(2,1);
		\draw[fill=cyan!50, draw=black] (2,0) rectangle ++(2,1);
		\draw[fill=cyan!50, draw=black] (-1,1) rectangle ++(2,1);
		\draw[fill=cyan!50, draw=black] (3,1) rectangle ++(2,1);
		\draw[fill=cyan!50, draw=black] (2,2) rectangle ++(2,1);
		\draw[fill=cyan!50, draw=black] (-2,2) rectangle ++(2,1);
		\draw[fill=cyan!50, draw=black] (1,3) rectangle ++(2,1);
		\draw[fill=cyan!50, draw=black] (3,3) rectangle ++(2,1);
		\end{tikzpicture}
		\caption{A restricted domino towers with $8$ blocks.}
		\label{fig:domino:dominoTowerRestricted}
	\end{minipage}\hspace*{1cm}
	\begin{minipage}{.42\textwidth}
		\centering
		\begin{tikzpicture}[scale=0.5]
		\draw[fill=cyan!50, draw=black] (0,0) rectangle ++(2,1);
		\draw[fill=cyan!50, draw=black] (2,0) rectangle ++(2,1);
		\draw[fill=cyan!50, draw=black] (0,1) rectangle ++(2,1);
		\draw[fill=cyan!50, draw=black] (3,1) rectangle ++(2,1);
		\draw[fill=cyan!50, draw=black] (2,2) rectangle ++(2,1);
		\draw[fill=cyan!50, draw=black] (-1,2) rectangle ++(2,1);
		\draw[fill=cyan!50, draw=black] (-1,3) rectangle ++(2,1);
		\draw[fill=cyan!50, draw=black] (2,3) rectangle ++(2,1);
		\draw[fill=cyan!50, draw=black] (4,2) rectangle ++(2,1);
		\draw[fill=cyan!50, draw=black] (1,4) rectangle ++(2,1);
		\end{tikzpicture}
		\caption{An unrestricted domino towers with $10$ blocks. }
		\label{fig:domino:dominoTowerUnrestricted}
	\end{minipage}
\end{figure}

In 2016, Brown generalized the problem to unrestricted towers made up of rectangles of width~$k$, which she called $k$-omino towers \cite{1608.01563}. She also introduced a variable~$b \ge 1$ for the number of blocks in the bottom row. The number of $k$-omino towers is $\binom{kn-1}{n-b}$. 

Brown also suggests that enumerating towers using rectangles of mixed widths could be interesting for other applications \cite[p.\ 17]{1608.01562}. In this paper we study this generalization by allowing rectangles with any width in a fixed finite list $S=(s_1,\ldots,s_m)$ of positive integers. We call this set of towers \emph{$S$-omino} towers. We additionally fix a list $(n_1,\ldots,n_m)$, where $n_i$ denotes the number of blocks of width $s_i$, and $b \ge 1$ the number of blocks in the bottom row. Furthermore, let $n:=n_1+\cdots+n_m \ge 1$ be the total number of blocks as before. We now state the first result of this paper. Note that for $S=(k)$ we, of course, recover the same formula as found by Brown. 
\begin{theorem}
	\label{domino:main}
	The number of $S$-omino towers with $n$ blocks of which $n_i$ have width~$s_i$, and $b$ blocks in the bottom row, which has to be convex, equals
	$$\binom{n}{n_1,\ldots,n_m}\binom{-1+\sum n_i s_i}{n-b}.$$
	Summing over all $b \in [n]$ we get
$$\binom{n}{n_1,\ldots,n_m} \binom{-1+\sum_{i=1}^m s_i n_i }{n-1}\cdot \, _2F_1\bigg(1,1-n;1+\sum_{i=1}^m (s_i-1) n_i ;-1\bigg),$$
where $_2F_1$ is the Gaussian hypergeometric function.
\end{theorem}

Note, that the heights of the blocks do not change the result, as we will explain in the next section. In particular, setting~$S=(1,k)$ for $k \ge 2$ corresponds to stacking~$k$-ominoes horizontally or vertically. In this paper we assume that $S$ does not contain duplicate entries for ease of notation. However, the methods would work and yield the same formula. Duplicate entries could be interpreted as having multiple distinguishable versions of dominoes with equal width. 

At the end of the paper we turn our attention to convex $k$-omino towers, which are defined as follows:
\begin{definition}
	\label{def:rowconvex}
	A tower is called \emph{column-convex} or \emph{row-convex} if all its columns or respectively rows are convex. Further, a tower is called \emph{convex} if it is both column- and row-convex.
\end{definition}
In 2016, Brown calculated the generating function for convex towers and asked whether row-convex towers can be enumerated as well \cite[p.\ 17]{1608.01562}.
\begin{definition}
	Let $g(n)$ be the number of row-convex $k$-omino towers made up of $n$ $k$-ominoes. We also define $f_\ell(n)$ to be the number of row-convex $k$-omino towers made up of $n$ $k$-ominoes resting on a platform of width $\ell k$. In other words the blocks on the bottom row need to rest on this platform, but the platform does not count towards the number of blocks.
\end{definition}
By adapting a method that Privman and \v{S}vraki\'{c} used in 1988 to calculate so-called fully directed compact lattice animals \cite{PhysRevLett.60.1107}, we calculate the ordinary generating functions~$G(z)=\sum_{n=0}^\infty g(n) z^n$ and~$F_\ell(z) = \sum_{n=0}^\infty f_\ell(n) z^n$.
\begin{theorem}
	\label{thm:rowconvex}
	We have:
	\begin{align*}
	G(z)=\sum_{\ell=1}^\infty z^\ell F_{\ell}(z),
	\end{align*}
	where
	\begin{align*}
	F_\ell(z) &= \bigg((1+kz)T_{1,\ell}+(kz^2-1)T_{2,\ell}+(k-1)z^3T_{3,\ell}\bigg)/\bigg((k-1)^2 z^5T_{2,3}\\
	&+(1-(2k-1)(1+z)z+k^2 z^3)T_{1,2}+(k-1) ((2k-1)z-1) z^3 T_{1,3}\bigg),\\
	T_{s,t}(z)&=A_s B_t - A_t B_s,\\
	A_\ell(z)&=\sum_{j=0}^\infty \frac{z^{\ell j} h_j }{\left(z;z\right)_j^2},\\
	B_\ell(z)&=\sum _{j=0}^{\infty} \frac{z^{\ell j} h_j }{\left(z;z\right)_j^2} \Bigg (\ell+\sum _{m=1}^j \left (1+\frac{2}{1-z^m}-\frac{1}{1+(k-1) z^m} \right) \Bigg ),\ \text{and}\\
	h_j(z)&:=z^{j(j+1)}\left((1-k)z;z\right)_j.
	\end{align*}	
\end{theorem}

The structure of the paper is as follows:
In Section~\ref{domino:sec:tools} we will introduce a few ideas and notation, which we need in later sections. In Section~\ref{domino:sec:lagrange} we prove Theorem~\ref{domino:main} using ordinary generating functions and the Lagrange inversion formula. In Section~\ref{domino:sec:bijection} we will turn the proof into an explicit bijection and show the connection to generalized Dyck paths \cite{Rukavicka2011}. Finally, in Section~\ref{domino:sec:rowconvex} we prove Theorem \ref{thm:rowconvex}.

\section{Representation of towers as sequences}\label{domino:sec:tools}
We can think of domino towers as being built by dropping single dominoes one by one. However, there may be multiple ways of building the same tower. Yet observe, that 
\begin{lemma}\label{domino:uniqueOrder}
	For any domino tower, there is a unique order $b_1,\ldots,b_n$ of its dominoes such that:
	\begin{enumerate}[i)]
		\item The tower can be built by dropping the dominoes in the order $b_1,\ldots,b_n$;
		\item For all $i\in [n-1]$, the left border of $b_{i+1}$ is strictly to the left of the right border of $b_{i}$.
	\end{enumerate}
\end{lemma}
\begin{proof}
	Consider the set of blocks, which have no other blocks above them, i.e., could be the last block according to condition i). Let $b'$ be the left-most block from this set and $b^\star$ be the block, that is dropped last and suppose, for a contradiction, that $b' \ne b^\star$. Then all blocks after $b'$ must be strictly to the left of $b'$ -- in particular $b^\star$ is to the left of $b'$. This contradicts the definition of $b'$, and therefore $b'$ is dropped last and we are done by induction on $n$.
\end{proof}

Hence, instead of enumerating towers, we can enumerate valid sequences of $x$-coordinates of the left borders of blocks $b_1,\ldots,b_n$. We fix $x_1 =0$ to keep this sequence unique. For example, $\langle 0,1,0,1,-2,-1,-3,-4\rangle$ is the sequence for the tower in Figure~\ref{fig:domino:dominoTowerRestricted}. We can define the set of valid sequences of $x$ coordinates as follows:
\begin{definition}\label{domino:def:Wb2}
	For $b\ge1$ we define $\mathcal{W}_b$ to be the set of sequences $\langle x_1,\ldots,x_n\rangle $ such that
	\begin{enumerate}[i)]
		\item $x_1 = 0$;
		\item For all $i\in [n-1]$ we have $x_{i+1} < x_{i}+2$;
		\item There exists a set $D \subset [n]$ containing $b$ numbers $1 = d_1 < d_2 < \dots < d_b$ such that \begin{enumerate}
			\item For all $j\in[b-1]$ we have $x_{d_{j+1}} + 2 = x_{d_j}$;
			\item For all $i \notin D$ with $i < d_{j+1}$ we have $x_i \ge x_{d_j}$;
			\item For all $i \notin D$ with $d_b < i$ we have $x_i + 2 > \min_{j<i} x_{j}$.
		\end{enumerate}
	\end{enumerate}
\end{definition}

Note, that by thinking about sequences of $x$-coordinates instead of towers, it is now clear that the heights of blocks do not change the number of towers. Only the~$x$-coordinates and order matter. Also note, that a tower is restricted if and only if the corresponding sequence does not have repeated consecutive entries. Using this insight we can now explain the relation between restricted and unrestricted towers:
\begin{corollary}\label{domino:4^n}
	The number of unrestricted domino towers made out of $n$ dominoes is equal to $4^{n-1}$.
\end{corollary}
\begin{proof}
	Consider the sequences in $\mathcal{W}:= \bigcup_b \mathcal{W}_b$ that have no repeated consecutive entries. We know from the introduction that there are $3^{n-1}$ such sequences of length~$n$. The corresponding ordinary generating function is therefore $f(x) = x+3x^2+9x^3+\cdots = \frac{x}{1-3x}$. Now, if we take such a sequence and replace every entry $x_i$ by a sequence $x_i,\ldots,x_i$ of arbitrary positive length, we get a sequence in $\mathcal{W}$, where repeated consecutive entries are allowed, i.e., a sequence corresponding to an unrestricted tower. This process is reversible: To recover the original sequence, we simply delete repeated consecutive entries. See Figure~\ref{fig:domino:substitution} for an illustration. This is an example of a \emph{substitution}\index{ordinary generating function!substitution} as defined by Flajolet \cite[Definition I.14]{flajolet}. In terms of the generating functions this procedure therefore corresponds to replacing $x$ with $x+x^2+x^3+\cdots=\frac{x}{1-x}$. We can now deduce the generating function for the unrestricted domino towers: 
	$$\frac{x}{1-3x} \leadsto \frac{\frac{x}{1-x}}{1-3\cdot \frac{x}{1-x}}=\frac{x}{1-4x},$$
	from which we can read off the number: $4^{n-1}$.	
\end{proof}
\begin{figure}
	\centering
	\begin{tikzpicture}[scale=0.5]
	\draw[fill=violet!50, draw=black] (0,0) rectangle ++(2,1);
	\draw[fill=cyan!50, draw=black] (2,0) rectangle ++(2,1);
	\draw[fill=cyan!50, draw=black] (-1,1) rectangle ++(2,1);
	\draw[fill=cyan!50, draw=black] (3,1) rectangle ++(2,1);
	\draw[fill=cyan!50, draw=black] (2,2) rectangle ++(2,1);
	\draw[fill=cyan!50, draw=black] (1,3) rectangle ++(2,1);
	\draw[fill=red!50, draw=black] (3,3) rectangle ++(2,1);
	
	\draw[thick,<->] (5.5,2.5) -- (7,2.5);
	
	\draw[fill=violet!50, draw=black] (8,0) rectangle ++(2,1);
	\draw[fill=violet!50, draw=black] (8,1) rectangle ++(2,1);
	\draw[fill=violet!50, draw=black] (8,2) rectangle ++(2,1);
	\draw[fill=violet!50, draw=black] (8,3) rectangle ++(2,1);
	\draw[fill=cyan!50, draw=black] (10,0) rectangle ++(2,1);
	\draw[fill=cyan!50, draw=black] (7,4) rectangle ++(2,1);
	\draw[fill=cyan!50, draw=black] (11,1) rectangle ++(2,1);
	\draw[fill=cyan!50, draw=black] (10,2) rectangle ++(2,1);
	\draw[fill=cyan!50, draw=black] (9,4) rectangle ++(2,1);
	\draw[fill=red!50, draw=black] (11,3) rectangle ++(2,1);
	\draw[fill=red!50, draw=black] (11,4) rectangle ++(2,1);
	\draw[fill=red!50, draw=black] (11,5) rectangle ++(2,1);
	\end{tikzpicture}
	\caption{Illustration of the substitution of $x$ with $\frac{x}{1-x}$. The corresponding sequences are
		$\langle 0,1,0,\textcolor{red}{1},\textcolor{violet}{-2},-1,-3\rangle$ and
		$\langle 0,1,0,\textcolor{red}{1,1,1},\textcolor{violet}{-2,-2,-2,-2},-1,-3\rangle$ respectively.}
	\label{fig:domino:substitution}
\end{figure}
We can generalize sequences $\mathcal{W}_b$ for $S$-omino towers by also keeping track of the widths $\len_i$. For that we redefine $\mathcal{W}_b$ as follows:

\begin{definition}\label{domino:def:Wb}
	We define $\mathcal{W}_b$ to be the set of sequences of pairs $\langle (x_1,\len_1),\ldots,(x_n,\len_n)\rangle $ with $n \ge b$ such that
	\begin{enumerate}[i)]
		\item $x_1 = 0$;
		\item For all $i\in [n-1]$ we have $x_{i+1} < x_{i}+\len_i$;
		\item There exists a set $D \subset [n]$ containing $b$ numbers $1 = d_1 < d_2 < \dots < d_b$ such that: \begin{enumerate}
			\item For $j\in[b-1]$ we have $x_{d_{j+1}} + \len_{d_{j+1}} = x_{d_j}$;
			\item For all $i \notin D$ with $i < d_{j+1}$ we have $x_i \ge x_{d_j}$;
			\item For all $i \notin D$ with $d_b < i$ we have $x_i + \len_i > \min_{j<i} x_{j}$.
		\end{enumerate}
	\end{enumerate}
\end{definition}

As we want to keep track of how many blocks of each length we have used, we define the weight of a sequence $t = \langle (x_1,\len_1),\ldots,(x_n,\len_n)\rangle \in \mathcal{W}$ as $\wei(t) := z^n y_{\len_1} y_{\len_2} \cdots y_{\len_n}$. Hence the exponent of $y_\len$ is the number of pairs in $t$ with second entry equal to $\len$ and the exponent of $z$ is the total number of pairs. We define the multivariate ordinary generating function of $\mathcal{W}_b$ with formal variables $z,y_1,y_2,y_3,\ldots$ as $W = \sum_{t \in \mathcal{W}} \wei(t)$. Lemma~\ref{domino:uniqueOrder} now immediately generalizes to:
\begin{lemma}\label{domino:uniqueOrderGeneralised}
	Fix $s_i,n_i$ and $b$ as before. Then there is a bijection between such $S$-omino towers and elements in $\mathcal{W}_b$ of weight $z^n \prod_i y_{s_i}^{n_i}$. Furthermore a tower is restricted if and only if there are no repeated consecutive elements in the corresponding sequence in $\mathcal{W}_b$. 
\end{lemma}
\begin{proof}
	Similarly to Lemma~\ref{domino:uniqueOrder}, the last element of the sequence must correspond to the left-most block of the tower, among the blocks that do not have any other blocks vertically above it. The statement follows from induction on $n$.
\end{proof}

We will often need to offset sequences of pairs of the form $(x,\len)$ horizontally, so we define
$$\langle (x_1,\len_1),\ldots,(x_n,\len_n) \rangle + \alpha :=\langle (x_1+\alpha,\len_1),\ldots,(x_n+\alpha,\len_n) \rangle.$$

Also, we define \textit{concatenation} of two sequences as follows:
$$\langle a_1,\ldots,a_n \rangle \mathbin \Vert \langle b_1,\ldots,b_m \rangle :=\langle a_1,\ldots,a_n, b_1,\ldots,b_m \rangle.$$

\section{Proof using the Lagrange inversion formula} \label{domino:sec:lagrange}

The main tool we use in this section is the following version of the Lagrange inversion formula \cite[Section 2.6]{miklosBonaHandbook}. Here $[x^n] G(x)$ denotes the coefficient of $x^n$ in the formal power series~$G(x)$.

\begin{proposition}[The Lagrange inversion formula \cite{miklosBonaHandbook}]
	\label{domino:lagrange}
	Let $Y(x)=x \Phi(Y)$, where $\Phi(Y)$ is a power series such that $\Phi(0)\ne 0$. Then for any power series $g(Y)$ and $n \ge 1$ we have
	$$[x^n] g(Y) = \frac{1}{n} [y^{n-1}] g'(y) (\Phi(y))^n.$$
\end{proposition}

We prove an immediate Corollary:
\begin{corollary}\label{domino:cor:lagrange}
	Let $Y(x)=x \Phi(Y)$, where $\Phi(Y)$ is a power series such that $\Phi(0)\ne 0$. Then for any power series $h(Y)$ and $n \ge 1$ we have:
	$$[x^n] x \frac{dY}{dx} h(Y) = [y^{n-1}] h(y) (\Phi(y))^n.$$
\end{corollary}
\begin{proof}
	Let $g(Y)=\int h(Y) dY$ and apply Proposition~\ref{domino:lagrange} on $g$.
	\begin{align*}
	&[x^n] x \frac{dY}{dx} h(Y)\\
	&=n\cdot [x^n] \int \frac{dY}{dx} h(Y) dx\\
	&=n\cdot [x^n] \int h(Y) dY \\
	&=[y^{n-1}] h(y)  (\Phi(y))^n.
	\end{align*} 
\end{proof}

The idea of the proof is to relate $\mathcal{W}_b$ to other sets of sequences.
\begin{definition}\label{domino:def:U}
	We define $\mathcal{U}_{\star}$ to be the set of sequences of pairs $\langle(x_1,\len_1),\ldots,(x_n,\len_n)\rangle$ with $n \ge 1$ such that:
	\begin{enumerate}[i)]
		\item $x_1=0$;
		\item For all $i>1$ we have $x_i \ge 1$ and $x_{i} < x_{i-1}+\len_{i-1}$
	\end{enumerate}
	We also define $\mathcal{U}_{1} := \mathrm{Seq}_{\ge 1}(\mathcal{U}_{\star})$, the set of sequences that are a concatenation of at least one sequence in $\mathcal{U}_{\star}$. For convenience we similarly define $\mathcal{U} := \mathrm{Seq}_{\ge 0}(\mathcal{U}_{\star}) = \{\langle \rangle\} \cup \mathcal{U}_1$, which also contains the empty sequence.
\end{definition}
\begin{lemma}\label{domino:lemma:splitU}
	Let $\langle(x_1,\len_1),\ldots,(x_n,\len_n)\rangle \in \mathcal{U}_\star$. Then there is exactly one choice of indices $d_1,\ldots,d_{\len_1}$ such that:
	\begin{itemize}
		\item $2=d_1 \le d_2 \le \ldots \le d_{\len_1} = n+1$;
		\item For all $j \in [\len_1-1]$, the subsequence
		$\langle(x_{d_j},\len_{d_j}),\ldots,(x_{d_{j+1}-1},\len_{d_{j+1}-1})\rangle -\alpha_j \in \mathcal{U}$, where $\alpha_j := \len_1-j$. In particular, if $d_j < d_{j+1}$ then $x_{d_j} = \alpha_j$.
	\end{itemize} 
\end{lemma}
\begin{proof}
	Define the indices $d_j^\star := \min \{i \ge 2 :x_i \le \alpha_j\}$ where $d_j^\star = n+1$ if such an $i$ does not exist. We now prove by induction on $j$ that $d_{j+1} = d_{j+1}^\star$ is the unique choice for indices that satisfy the conditions in the lemma.
	Clearly for $j=0$ we are done, because~$d_1=2=d_1^\star$. Similarly for $j=\len_1-1$ we have $d_{\len_1}=n+1=d_{\len_1}^\star$. Now for $j \in [\len_1-2]$ we can assume the induction hypothesis for $j-1$, i.e., that $d_{j}=d_{j}^\star$. For a contradiction, we consider the cases~$d_{j+1}>d_{j+1}^\star$ and $d_{j+1}<d_{j+1}^\star$ separately:
	
	\textbf{Case 1:}
	If $d_{j+1}>d_{j+1}^\star \ge d_j^\star = d_j$ then by definition of $d_{j+1}^\star$ we have $x_{d_{j+1}^\star} \le \alpha_{j+1}$. However by definition of $\mathcal{U}_\star$ and the fact that $\langle(x_{d_j},\len_{d_j}),\ldots,(x_{d_{j+1}-1},\len_{d_{j+1}-1})\rangle -\alpha_j \in \mathcal{U}$ we have $x_{d_{j+1}^\star} \ge \alpha_j = \alpha_{j+1} + 1$, a contradiction. 
	
	\textbf{Case 2:}
	If $d_{j+1}<d_{j+1}^\star$ then by definition of $d_{j+1}^\star$ we have $x_{d_{j+1}} > \alpha_{j+1}$. However for $k := \max \{p:d_p = d_{j+1}\}$ we have $d_k < d_{k+1}$ and therefore $x_{d_{j+1}} = x_{d_k} = \alpha_k \le \alpha_{j+1}$, a contradiction.
	
	Hence $d_{j+1}=d_{j+1}^\star$ and we are done by induction on $j$.
\end{proof}

\begin{example}\label{domino:example:splitU}
The sequence $\langle (0,4),(3,2),(3,2),(1,2),(2,2) \rangle \in \mathcal{U}_\star$ corresponds to the tower in Figure \ref{fig:domino:splitU}. The indices $d_i$ in this example are: $d_1=2$, $d_2=4$, $d_3=4$ and $d_4=6$. We have $\langle (3,2),(3,2) \rangle - 3 \in \mathcal{U}$ drawn in blue, $\langle \rangle - 2 \in \mathcal{U}$, and finally $\langle (1,2),(2,2) \rangle - 1 \in \mathcal{U}$ drawn in red.
\begin{figure}[ht]
	\centering
	\begin{tikzpicture}[scale=0.5]
	\draw[fill=cyan!50, draw=black] (0,0) rectangle ++(4,1);
	\draw[fill=violet!50, draw=black] (3,1) rectangle ++(2,1);
	\draw[fill=violet!50, draw=black] (3,2) rectangle ++(2,1);
	\draw[fill=red!50, draw=black] (1,1) rectangle ++(2,1);
	\draw[fill=red!50, draw=black] (2,3) rectangle ++(2,1);
	\end{tikzpicture}
	\caption{Illustration of Example \ref{domino:example:splitU}.}
	\label{fig:domino:splitU}
\end{figure}
\end{example}

We have just shown that every sequence in $\mathcal{U}_\star$ can be built by concatenating $\langle (0,\len) \rangle$ with $\len-1$ sequences in $\mathcal{U}$, offset by $\len-1,\len-2,\ldots,1$ respectively and that this construction is unique. Similarly every sequence in $\mathcal{U}_1$ can be constructed uniquely using $\len$ sequences in $\mathcal{U}$. This motivates the following equivalent definition of $\mathcal{U}$. 
\begin{definition}\label{domino:def:U0}
	We define $\mathcal{U}$ to be the minimal set such that:
	\begin{enumerate}[i)]
		\item $\langle \rangle \in \mathcal{U}$;
		\item $\mathcal{U}$ is closed under the following procedure:\begin{enumerate}[a)]
			\item Pick any $\len \in \mathbb{N}$;
			\item Pick any $\len$ elements $c_0,\ldots,c_{\len-1}$ in $\mathcal{U}$;
			\item Then $\langle(0,\len)\rangle \mathbin\Vert c_{\len-1}+\len-1\mathbin\Vert c_{\len-2}+\len-2 \mathbin\Vert \cdots \mathbin\Vert c_0+0 \in \mathcal{U}$.
		\end{enumerate}
	\end{enumerate}
\end{definition}
We also define two other sets $\mathcal{X}$ and $\mathcal{V}$ in a similar manner.
\begin{definition}\label{domino:def:XV}
	We define $\mathcal{X}_{\len}$ to be the set of sequences
	$$\mathcal{X}_{\len} := \{u-\alpha: 0 \le \alpha < \len \wedge u\in \mathcal{U}_\star \wedge u \text{ starts with }(0,\len)\}.$$
	Let $\mathcal{X} = \bigcup_{\len \in \mathbb{N}} \mathcal{X}_{\len}$ be the union over all possible lengths $\len$. Further, we define $\mathcal{V}$ to be the minimal set such that:
	\begin{enumerate}[i)]
		\item $\langle \rangle \in \mathcal{V}$;
		\item $\mathcal{V}$ is closed under the following procedure:\begin{enumerate}[a)]
			\item Pick any $\len \in \mathbb{N}$;
			\item Pick any elements  $v \in \mathcal{V}$ and $x \in \mathcal{X}_{\len}$ starting with, say, $(-\alpha,\len)$;
			\item Then $x \mathbin\Vert (v-\alpha) \in \mathcal{V}$.
		\end{enumerate}
	\end{enumerate}
\end{definition}

\begin{lemma}\label{domino:lemma:splitW}
	$\mathcal{W}_b$ is related to $\mathcal{X}$ and $\mathcal{V}$ as follows:
	\begin{enumerate}[i)]
		\item For all $b\ge 2$ we have a weight-preserving bijection $\mathcal{W}_b \leftrightarrow \mathcal{U}_1 \times  \mathcal{W}_{b-1}$;
		\item We have a weight-preserving bijection $\mathcal{W}_1 \leftrightarrow \mathcal{U}_\star \times V$.
	\end{enumerate}
\end{lemma}
\begin{proof}
	Consider any element $\langle (x_1,\len_1),\ldots,(x_n,\len_n)\rangle \in \mathcal{W}_b$. For $b\ge 2$ we know from Definition~\ref{domino:def:Wb} that there exists an index $1 < d_2 = \min \{i \ge 2 : x_i < 0 \}$ such that we have $$\langle (x_1,\len_1),\ldots,(x_{d_2-1},\len_{d_2-1})\rangle \in \mathcal{U}_1\text{ and }\langle (x_{d_2},\len_{d_2}),\ldots,(x_{n},\len_{n})\rangle + \len_{d_2} \in \mathcal{W}_{b-1}.$$

	For $b=1$ we let $d = \min \{i \ge 2 : x_i \le 0 \}$ and $d = n+1$ if such an $i$ does not exist. Then $$\langle (x_1,\len_1),\ldots,(x_{d-1},\len_{d-1}) \rangle \in \mathcal{U}_\star\text{ and }\langle (x_d,\len_d),\ldots,(x_{n},\len_{n}) \rangle \in \mathcal{V}.$$
	In both cases the function has an inverse: Concatenate both parts back together.
\end{proof}

\begin{example}\label{domino:example:splitW}
The sequence $\langle (0,2),(1,2),(2,2),(-2,2),(-1,2),(-3,2),(-4,2)\rangle \in \mathcal{W}_2$ corresponds to the tower in Figure \ref{fig:domino:splitW}. We have $\langle (0,2),(1,2),(2,2)\rangle \in \mathcal{U}_1$ drawn in blue, $\langle (-2,2),(-1,2)  \rangle + 2 \in \mathcal{U}_\star$ drawn in violet, and finally $\langle (-3,2),(-4,2) \rangle + 2 \in \mathcal{V}$ drawn in red.
\begin{figure}[ht]
	\centering
	\begin{tikzpicture}[scale=0.5]
	\draw[fill=cyan!50, draw=black] (0,0) rectangle ++(2,1);
	\draw[fill=cyan!50, draw=black] (1,1) rectangle ++(2,1);
	\draw[fill=cyan!50, draw=black] (0,2) rectangle ++(2,1);
	\draw[fill=violet!50, draw=black] (-2,0) rectangle ++(2,1);
	\draw[fill=violet!50, draw=black] (-1,3) rectangle ++(2,1);
	\draw[fill=red!50, draw=black] (-3,1) rectangle ++(2,1);
	\draw[fill=red!50, draw=black] (-4,2) rectangle ++(2,1);
	\end{tikzpicture}
	\caption{Illustration of Example \ref{domino:example:splitW}.}
	\label{fig:domino:splitW}
\end{figure}
\end{example}

Using the same weights as for $W$, let $U,U_1,U_\star,X_\len,X,V$ be the multivariate ordinary generating functions of $\mathcal{U}, \mathcal{U}_1, \mathcal{U}_\star, \mathcal{X}_\len, \mathcal{X}, \mathcal{V}$ respectively.

\begin{theorem}\label{domino:thm:long}
	The ordinary generating functions satisfy the following equations:
	\begin{enumerate}[a)]
		\item \label{domino:longtheorem:part_1} $U = 1 + U_1$ and $U_1 = U_\star \cdot U = \sum_{\len\in \mathbb{N}} z y_\len U^\len$
		\item \label{domino:longtheorem:part_2} $X_{\len} =\len  z y_\len   U^{\len-1}$, $X = \sum_{\len\in \mathbb{N}} X_{\ell}$, and $V = 1 + X \cdot V$
		\item \label{domino:longtheorem:part_3} $W_b = U_\star \cdot V \cdot U_1^{b-1}$
		\item \label{domino:longtheorem:part_4} $z \frac{dU_1}{dz} = U_1 + X \cdot z \frac{dU_1}{dz}$
		\item \label{domino:longtheorem:part_5} $z \frac{dU_1}{dz} = U_1 \cdot V$
		\item \label{domino:longtheorem:part_6} $(1+U_1) \cdot W_b = U_1^{b-1} z \frac{dU_1}{dz} $
	\end{enumerate}
\end{theorem}
\begin{proof}
	Parts~\ref{domino:longtheorem:part_1}, \ref{domino:longtheorem:part_2} and \ref{domino:longtheorem:part_3} follow from Lemma~\ref{domino:lemma:splitU}, Definition~\ref{domino:def:XV} and Lemma~\ref{domino:lemma:splitW} respectively. For~\ref{domino:longtheorem:part_4} note, that $z \frac{dU_1}{dz}$ is the ordinary generating function of $\Theta \mathcal{U}_1$, i.e., the set $\mathcal{U}_1$, where one element is marked. We can define $\Theta \mathcal{U}_1 := \{(u,k) : u \in \mathcal{U}_1, k\in [|u|]\}$.
	We now describe a bijection $f$ between $\Theta \mathcal{U}_1$ and $\mathcal{U}_1 + X \times \Theta \mathcal{U}_1$. Let $(u,k) \in \Theta \mathcal{U}_1$. Now note, that there are two cases:
	\begin{enumerate}
		\item For $k=1$ we simply define $f((u,k)) := u \in \mathcal{U}_1$;
		\item For $k\ge 2$ we know that by Definition~\ref{domino:def:U0} there exists $\len \in \mathbb{N}$ and sequences $c_0,\ldots,c_{\len-1} \in \mathcal{U}$ and $2 = d_1\le \cdots\le d_{\len+1} = n+1$ such that $$u = \langle(0,\len)\rangle \mathbin\Vert (c_{\len-1}+\len-1)\mathbin\Vert (c_{\len-2}+\len-2) \mathbin\Vert \cdots \mathbin\Vert (c_0+0),$$
		where $|c_{\len-i}|=d_{i+1}-d_i$. As $k\ge 2$ there exists exactly one $p \in [\len]$ such that $d_p \le k < d_{p+1}$. Hence $(c_{\len-p},k-d_p+1) \in \Theta \mathcal{U}_1$. Let 
		\begin{align*}
		x = \langle (-\len+p,\len)\rangle &\mathbin\Vert (c_{\len-1}+p-1)\mathbin\Vert \cdots \mathbin\Vert (c_{\len-p+1}+1) \\
		&\mathbin\Vert  (c_{\len-p-1}) \mathbin\Vert \cdots \mathbin\Vert (c_0-\len+p+1).
		\end{align*}
		Then $x \in \mathcal{X}_{\len}$. We define $f((u,k)) = (x,(c_{\len-p},k-d_p+1)) \in \mathcal{X} \times \Theta \mathcal{U}_1$. This function is invertible, because $x$ stores the variable $p$, which enables us to undo the shifts and reinsert $c_{\len-p}$ into $u$ at the correct position.
	\end{enumerate}
	
	For \ref{domino:longtheorem:part_5} note, that from \ref{domino:longtheorem:part_2} and \ref{domino:longtheorem:part_4} follows $ \Theta \mathcal{U}_1 = \mathcal{U}_1 \times\mathrm{Seq}_{\ge 0}(\mathcal{X}) = \mathcal{U}_1 \times \mathcal{V}$. Finally \ref{domino:longtheorem:part_6} follows from \ref{domino:longtheorem:part_1},  \ref{domino:longtheorem:part_3} and \ref{domino:longtheorem:part_5}.
\end{proof}

\begin{proof}[Proof of Theorem~\ref{domino:main}]
	
	We consider the multivariate power series $W_b$ in $z,y_1,y_2,y_3,\ldots$ as a power series in $z$ with coefficients in the ring of multivariate formal power series in $y_1,y_2,y_3,\ldots$ and use Proposition~\ref{domino:cor:lagrange}:
	\begin{align*}
	[z^n]W_b &= [z^n] z \frac{dU_1}{dz} \frac{U_1^{b-1}}{1+U_1} \\ \displaybreak[0]
	&=  [u^{n-1}] \frac{u^{b-1}}{1+u} \bigg (\sum_{i \in [m]} y_{s_i} (1+u)^{s_i} \bigg)^n\\\displaybreak[0]
	&=  [u^{n-b}] \frac{1}{1+u} \bigg (\sum_{i \in [m]} y_{s_i} (1+u)^{s_i} \bigg)^n.
	\end{align*}
	Now we also fix the number of occurrences of length $s_i$ to be $n_i$:
	\begin{align*}
	[y_{s_1}^{n_1} \cdots y_{s_m}^{n_m}] [z^n]W_b &= [y_{s_1}^{n_1} \cdots y_{s_m}^{n_m}] [u^{n-b}] \frac{1}{1+u} \bigg (\sum_{i \in [m]} y_{s_i} (1+u)^{s_i} \bigg)^n\\\displaybreak[0]
	&=[u^{n-b}] \frac{1}{1+u} \binom{n}{n_1,\ldots,n_m}  \prod_{i=1}^m (1+u)^{s_i n_i}\\\displaybreak[0]
	&= [u^{n-b}]\binom{n}{n_1,\ldots,n_m} (1+u)^{-1+\sum_{i=1}^m s_i n_i}\\\displaybreak[0]
	&= \binom{n}{n_1,\ldots,n_m}\binom{-1+\sum_{i=1}^m s_i n_i }{n-b}.
	\end{align*}
	Now, summing over all $b\in [n]$ we can express the total number of $S$-omino towers for given~$(n_1,\ldots,n_m)$ in terms of the Gaussian hypergeometric function $\,_2F_1$.
	\begin{align*}
	& [z^n y_{s_1}^{n_1} \cdots y_{s_m}^{n_m}] \sum_{b=1}^n W_b\\
	&=\sum_{b=1}^n \binom{n}{n_1,\ldots,n_m}\binom{-1+\sum_{i=1}^m s_i n_i }{n-b}\\
	&=\binom{n}{n_1,\ldots,n_m} \binom{-1+\sum_{i=1}^m s_i n_i }{n-1}\sum_{b=0}^{n-1}  \frac{(n - 1)!( -n+\sum_{i=1}^m s_i n_i)!}{(n - 1 - b)!(b-n+\sum_{i=1}^m s_i n_i)!} \\
	&=\binom{n}{n_1,\ldots,n_m} \binom{-1+\sum_{i=1}^m s_i n_i }{n-1} \sum_{b=0}^{n-1} \frac{(1)_b(1-n)_b}{(1-n+\sum_{i=1}^m s_i n_i)_b} \frac{(-1)^b}{b!}\\
	&=\binom{n}{n_1,\ldots,n_m} \binom{-1+\sum_{i=1}^m s_i n_i }{n-1}\cdot \, _2F_1\bigg(1,1-n;1+\sum_{i=1}^m (s_i-1) n_i ;-1\bigg).
	\end{align*}
	This completes the proof of Theorem~\ref{domino:main}.
\end{proof}

\begin{remark}
	We can find closed formulas for the other sets analogously. For example for $s:=\sum n_i s_i$:
	\begin{align*}
	[y_{s_1}^{n_1} \cdots y_{s_m}^{n_m}] [z^n]V 
	&= \binom{s }{n_1,\ldots,n_m,s-n}.
	\end{align*}
	and 
	\begin{align*}
	[y_{s_1}^{n_1} \cdots y_{s_m}^{n_m}][z^n]U
	&=\frac{1}{s+1} \binom{s+1}{n_1,\ldots,n_m,s+1-n}.\\
	\end{align*}
\end{remark}

\section{Bijective proof} \label{domino:sec:bijection}

In this section we give an explicit bijection between $\mathcal{U}$ and generalized Dyck paths and then extend it to a bijection between $\mathcal{W}_b$ and $\mathcal{D}_{W_b}$.
\begin{definition} \label{domino:def:DU}
	We define the weighted set of generalized Dyck paths as
	$$\mathcal{D}_{U} := \bigg \{\langle u_1,u_2,\ldots,u_N\rangle : \forall i\ u_i \ge -1 \wedge \forall j\ \sum_{i< j} u_i \ge 0 \wedge u_N=-1 \wedge \sum_{i \in [N]} u_i= -1\bigg\}$$
 with weight function $\wei(\langle u_1,u_2,\ldots,u_N\rangle) := \prod_{i\in [N]} \begin{cases}z y_{u_i+1},&\text{if $u\ge 0$;}\\1,&\text{if $u=-1$.}\end{cases}$
\end{definition}
See \cite{Rukavicka2011} for an alternative, but equivalent definition of generalized Dyck paths.

\begin{lemma}
	We have a weight-preserving bijection $f_U$ between $\mathcal{U}$ and $\mathcal{D}_U$.
\end{lemma}
\begin{proof}
	We define $f_U(\langle \rangle):=\langle -1 \rangle$ and given a sequence $u = \langle (x_1,\len_1),\ldots,(x_n,\len_n) \rangle \in \mathcal{U}_1$, we define 
	$$f_U(u) := \big(\mathbin\Vert_{i=1}^{n-1} \langle \len_i-1, \underbrace{-1,-1,\ldots,-1}_{x_i+\len_i-1-x_{i+1}\text{ times}} \rangle \big) \mathbin\Vert \langle \len_n-1, \underbrace{-1,-1,\ldots,-1}_{x_n+\len_n\text{ times}} \rangle \in \mathcal{D}_U.$$
	Note, that $f_U$ is clearly weight-preserving, as for every pair $(x,\len) \in u$ we have got exactly one copy of $\len-1$ in $f_U(u)$, both of which account for the same weight: $z y_\len$. Also note that we can invert this function as the original $x$-coordinates of $u$ are prefix-sums of $f_U(u)$.
\end{proof}

As we have proven Theorem~\ref{domino:thm:long} bijectively, we already know how we can relate sequences in $\mathcal{V}$ and $\mathcal{W}$ with sequences in $\mathcal{U}$. We now reuse the ideas from the previous section and replace $\mathcal{U}$ with $\mathcal{D}_U$. Also, we use Raney's Lemma \cite{1609.05988} instead of the Lagrange inversion formula.
\begin{lemma}[Version of Raney's Lemma]
	\label{domino:lemma:cycle}
	For any sequence of integers $\langle a_1,\ldots,a_m\rangle $ with $a_i \ge -1$ and $\sum a_i = -1$, there exists exactly one $r \in [m]$ with the property that all proper partial sums, or in other words, the totals of all proper prefixes, of $\langle a_{r+1},\ldots,a_m,a_1,\ldots,a_{r}\rangle$ are non-negative. Note, that then we must have $a_r = -1$, of course, as $a_{r+1}+\cdots+a_m+a_1+\cdots+a_{r-1} \ge 0$, but $\sum a_i = -1$.
\end{lemma}
We now generalize the idea from Dyck paths to the following sets $\mathcal{V}$ and $\mathcal{W}$:
\begin{definition} \label{domino:def:DV}
	Define \begin{align*}
	\mathcal{D}_{V} &:= \bigg\{\langle u_1,u_2,\ldots,u_N\rangle : \forall i\ u_i \ge -1 \wedge \sum_{i \in [N]} u_i= -1\wedge u_N=-1\bigg \} \text{ and }\\
	\mathcal{D}_{W_b} &:= \bigg\{\langle u_1,u_2,\ldots,u_N\rangle : \forall i \in [b]\ u_i \ge 0 \wedge \sum_{i \in [N]} u_i= -b\wedge u_N=-1\bigg \},
	\end{align*} 
	where the weight of a sequence is $\wei(\langle u_1,u_2,\ldots,u_N\rangle) := \prod_{i\in [N]} \begin{cases}z y_{u_i+1},&\text{if $u\ge 0$;}\\1,&\text{if $u=-1$.}\end{cases}$
	
\end{definition}
We can easily enumerate elements in $V$ and $W_b$ with given weight.
\begin{lemma}
	Fix a weight $w = z^n \prod_{i \in [m]} y_{s_i}^{n_i}$ with $n = \sum_{i \in [m]} n_i$ and let $s = \sum_{i\in[m]} n_i s_i$. Then the number of elements in $\mathcal{D}_V$ with weight $w$ equals
	$$\binom{s}{n_1,n_2,\ldots,n_m,s-n}.$$	
	Also, the number of elements in $\mathcal{D}_{W_b}$ with weight $w$ equals
	$$\binom{n}{n_1,\ldots,n_m}\binom{s-1 }{n-b}.$$	
\end{lemma}
\begin{proof}
	For the first part note, that every element in $\mathcal{D}_{V}$ with weight $w$ has $s-n+1$ copies of $-1$ and total length $N=s+1$. The number of sequences follows from the fact, that only $u_N$ is fixed to be $-1$. 

For the second part note, that now every element in $\mathcal{D}_{W_b}$ with weight $w$ has $s-n+b$ copies of $-1$ and total length $N=s+b$. We have $u_N =-1$ and~$u_i \ne -1$ for $i\in [b]$. The number of elements in $\mathcal{D}_{W_b}$ follows by first considering the order and then the position of the non-negative integers.
\end{proof}

\begin{corollary}
	Define
	$$\mathcal{D}_{Y} := \{(v,k):v = \langle u_1,u_2,\ldots,u_N\rangle \in \mathcal{D}_{U}\wedge k\in [N] \wedge u_k=-1\},$$
	which can be thought of as the set of sequences in $\mathcal{D}_U$ where one copy of $-1$ is marked. Then there is a bijection $f_{V_2}$ between $\mathcal{D}_Y$ and $\mathcal{D}_{V}$.
\end{corollary}
\begin{proof}
	Let $f_{V_2}$ be the following bijection: Take an element $(v,k) \in \mathcal{D}_Y$ and let $f_{V_2}((v,k)):=\langle u_{k+1}, \ldots, u_N, u_1,\ldots,u_k \rangle \in \mathcal{D}_V$. In other words we have moved the marked element of an element in $\mathcal{D}_Y$ to the end by doing a cyclic rotation of the sequence. This can be undone using Raney's Lemma.
\end{proof}

\begin{lemma}
	We have a weight-preserving bijection $f_{V_1}$ between $\mathcal{V}$ and $\mathcal{D}_{Y}$. Then $f_V := f_{V_2} \circ f_{V_1}$ is a weight-preserving bijection between $\mathcal{V}$ and $\mathcal{D}_V$.
\end{lemma}
\begin{proof}
	We construct an explicit bijection $f_{V_1}: \mathcal{V} \rightarrow \mathcal{D}_{Y}$ as follows. First let $f_{V_1}(\langle \rangle) := (\langle -1 \rangle,1) $ and now consider any non-empty $v \in \mathcal{V}$. Then by definition $\exists \len \in \mathbb{N}, \alpha \in \{0,\ldots,\len-1\},  x \in \mathcal{X}_{\len}, v' \in \mathcal{V}$, where $x$ starts with $(-\alpha,\len)$, such that $v = x \mathbin\Vert (v'-\alpha)$.
	By definition of $\mathcal{X}_{\len}$ we can find $c_1,\ldots,c_{\len-1} \in \mathcal{U}$ such that $x = \big(\langle(0,\len)\rangle \mathbin\Vert c_{\len-1}+\len-1\mathbin\Vert c_{\len-2}+\len-2 \mathbin\Vert \cdots \mathbin\Vert c_{1}+1 \big) - \alpha$. Hence we have
	$$v = \big(\langle(0,\len)\rangle \mathbin\Vert c_{\len-1}+\len-1\mathbin\Vert c_{\len-2}+\len-2 \mathbin\Vert \cdots \mathbin\Vert c_{1}+1 \mathbin\Vert v' \big) - \alpha.$$
	
	Let $(u,k) := f_{V_1}(v')$. We define $k':=k+1+|f_U(c_{\len-1})|+\cdots+|f_U(c_{\alpha+1})|$ and 
	$$f_{V_1}(v):=  \big( \langle \len -1\rangle \mathbin\Vert f_U(c_{\len-1}) \mathbin\Vert \cdots \mathbin\Vert f_U(c_{\alpha+1})\mathbin\Vert u \mathbin\Vert f_U(c_{\alpha}) \mathbin\Vert \cdots \mathbin\Vert f_U(c_{1}),k'\big). $$
	Note that $k'$ is chosen such that the marked element in $u$ remains marked.
\end{proof}

\begin{lemma}
	We have a weight-preserving bijection $f_W$ between $\mathcal{W}_b$ and $\mathcal{D}_{W_b}$.
\end{lemma}
\begin{proof}
	From Lemma~\ref{domino:lemma:splitW} we know that we can split an element $w \in \mathcal{W}_b$ into two parts. For $b=1$ we have $w = u\mathbin \Vert v$ where $u\in\mathcal{U}_\star, v\in \mathcal{V}$ and define $f_W(w) := f_U(u)\mathbin \Vert f_V(v)$. For $b\ge 2$ we have $w = u\mathbin \Vert w' - \len$ where $u\in\mathcal{U}_1, w'\in \mathcal{W}_{b-1}$ and $u$ starts with $(0,\len)$. Here we could choose to define $f_W(w) := f_U(u)\mathbin \Vert f_W(w')$, but it would change the definition of $\mathcal{D}_W$ and make its enumeration more complicated. So we instead proceed as follows: Let $f_W(w') = \langle a_1,\ldots,a_n\rangle$. Then define $f_W(w) := \langle a_1,\ldots,a_{b-1}\rangle \mathbin \Vert f_U(u)\mathbin \Vert \langle a_b,\ldots,a_n\rangle$. We can undo the function in both cases, as we know that the sequence $f_U(u)$ sums to $-1$. $f_W$ is weight-preserving, because $f_U$ and $f_V$ are.
\end{proof}

\section{Row-convex k-omino towers} \label{domino:sec:rowconvex}

In this section we consider row-convex $k$-omino towers, as defined in Definition~\ref{def:rowconvex}. By conditioning on the width of the bottom row, we see that $f$ and $g$ are related by the equation
$$g(n)=\sum_{\ell=1}^{n} f_\ell(n-\ell)\text{, or equivalently }G(z)=\sum_{\ell=1}^\infty z^\ell F_{\ell}(z).$$
To improve readability, from now on we write $F_\ell(z)$ as $F_\ell$ and similarly for $G(z)$ and $h_n(z),\alpha(z)$ and~$\beta(z)$ which are yet to be introduced. If the platform has width $\ell$ and the bottom row consists of $i$ blocks, there are $(\ell+2-i)k-1$ positions the blocks in the row above could take such that the row is convex and they do not fall off the sides. This can be seen by an argument similar to the one given by Brown \cite[Proposition 2.5]{1608.01562}. We immediately find the recurrence:
\begin{align}
f_\ell(n)&=\sum_{i=1}^{\ell+1} \bigg ((\ell+2-i)k-1\bigg )f_i(n-i)\text{, for }n\ge 1\text{, where we define} \label{recf}\\
f_\ell(0)&=1\text{ and }f_\ell(n)=0\text{, for }n<0. \nonumber
\end{align}
We now reduce this recurrence in $f$ to a much simpler one and then prove
\begin{lemma}
	$F_\ell$ satisfies the following recurrence relation and boundary conditions:
	\begin{equation}F_{\ell+2}-2 F_{\ell+1} + F_\ell=z^{\ell+2} F_{\ell+2}+ (k-1)z^{\ell+3} F_{\ell+3}, \label{recF}\end{equation}
	\begin{equation}
	\begin{aligned}
	F_1 &=1+(2k-1 )z F_1+(k-1)z^2 F_2, \label{boundary}\\
	F_2 &= 1 + (3k-1)z F_1 + (2k-1)z^2 F_2 + (k-1)z^3 F_3.
	\end{aligned}
	\end{equation}
\end{lemma}
\begin{proof}
	
	First, we calculate
	$$f_{\ell+1}(n)-f_\ell(n)=(k-1) f_{\ell+2}(n-\ell-2)+ k \sum_{i=1}^{\ell+1} f_i(n-i)$$
	and then use this result twice as follows:
	\begin{align*}
	&f_{\ell+2}(n)-2f_{\ell+1}(n)+f_\ell(n)\\
	&=\big (f_{\ell+2}(n)-f_{\ell+1}(n)\big )-\big (f_{\ell+1}(n)-f_\ell(n)\big )\\
	&= (k-1)f_{\ell+3}(n-\ell-3) + f_{\ell+2}(n-\ell-2).
	\end{align*}
	The corresponding recurrence in terms of $F_\ell$ is 
	\begin{align*}
	&F_{\ell+2}-2 F_{\ell+1} + F_\ell\\
	=&\sum_{n=0}^\infty \bigg( f_{\ell+2}(n)-2f_{\ell+1}(n)+f_\ell(n) \bigg) z^n\\
	=&\sum_{n=0}^\infty \bigg( f_{\ell+2}(n-\ell-2) + (k-1)f_{\ell+3}(n-\ell-3) \bigg) z^n\\
	=&z^{\ell+2} F_{\ell+2}+ (k-1)z^{\ell+3} F_{\ell+3}.
	\end{align*}
	
	The boundary conditions are obtained by plugging in $\ell=1$ and $\ell=2$ into \eqref{recf}.
\end{proof}

To solve the recurrence \eqref{recF}, we first guess that there is a solution of the form $$\sum_{j=0}^\infty \frac{z^{\ell j}h_j(z)}{\left(z;z\right)_j^2}$$ and then determine an $h_j(z)$ such that the recurrence relation holds. Here $(a;q)_n=\prod_{i=0}^{n-1}(1-a q^i)$ denotes the $q$-Pochhammer symbol. That this method works is not surprising:
In 1988, Privman and \v{S}vraki\'{c} successfully found an exact generating function for \emph{fully directed compact lattice animals} using this approach \cite{PhysRevLett.60.1107}. The two problems are related, as there is a bijection between the set of objects they were considering and \emph{restricted} row-convex domino towers. The number of dominoes in the bottom row maps to the number of compact sources of the directed animal. For an illustration of this bijection see Figure~\ref{fig:directedanimals}.
\begin{figure}[ht]
	\centering
	\begin{tikzpicture}[scale=0.5]
	\draw[fill=cyan!50, draw=black] (0,0) rectangle ++(2,1);
	\draw[fill=violet!50, draw=black] (-1,1) rectangle ++(2,1);
	\draw[fill=cyan!50, draw=black] (1,1) rectangle ++(2,1);
	\draw[fill=cyan!50, draw=black] (2,2) rectangle ++(2,1);
	\draw[fill=violet!50, draw=black] (0,2) rectangle ++(2,1);
	\draw[fill=pink!50, draw=black] (-1,3) rectangle ++(2,1);
	\draw[fill=violet!50, draw=black] (1,3) rectangle ++(2,1);
	\draw[fill=violet!50, draw=black] (2,4) rectangle ++(2,1);
	
	\draw[thick,<->] (4.5,1.5) -- (6,1.5);
	
	\draw[fill=cyan!50, draw=black] (7,0) rectangle ++(1,1);
	\draw[fill=cyan!50, draw=black] (8,0) rectangle ++(1,1);
	\draw[fill=cyan!50, draw=black] (9,0) rectangle ++(1,1);
	\draw[fill=violet!50, draw=black] (7,1) rectangle ++(1,1);
	\draw[fill=violet!50, draw=black] (8,1) rectangle ++(1,1);
	\draw[fill=violet!50, draw=black] (9,1) rectangle ++(1,1);
	\draw[fill=violet!50, draw=black] (10,1) rectangle ++(1,1);
	\draw[fill=pink!50, draw=black] (8,2) rectangle ++(1,1);
	\end{tikzpicture}
	\caption{Illustration of the bijection between \emph{restricted} row-convex domino towers and fully directed compact lattice animals}
	\label{fig:directedanimals}
\end{figure}

After adapting their method we end up with two solutions $A_\ell $ and $B_\ell$, which we now check:
\begin{lemma}
	Two solutions of \eqref{recF} are:
	\begin{align*}
	A_\ell&:=\sum_{j=0}^\infty \frac{z^{\ell j} h_j }{\left(z;z\right)_j^2}\text{ and}\\
	B_\ell&:=\sum _{j=0}^{\infty} \frac{z^{\ell j} h_j }{\left(z;z\right)_j^2} \left (\ell+\sum _{m=1}^j \left (1+\frac{2}{1-z^m}-\frac{1}{1+(k-1) z^m} \right) \right )\text{, where}\\
	h_j&:=z^{j(j+1)}\left((1-k)z;z\right)_j.
	\end{align*}
\end{lemma}
\begin{proof}
	First we calculate the ratio
	$$\frac{h_j}{h_{j-1}} = \frac{z^{j(j+1)} }{z^{(j-1)j}} \frac{\prod_{i=1}^j (1-(1-k)z^i)}{ \prod_{i=1}^{j-1} (1-(1-k)z^i)} = z^{2 j} (1+(k-1) z^j).$$
	Then we show that the recurrence holds
	\begin{align*}
	&A_{\ell+2}-2 A_{\ell+1} + A_\ell\\
	&=\sum_{j=0}^\infty \frac{ h_j }{\left(z;z\right)_j^2} \left(z^{(\ell+2) j}-2z^{(\ell+1) j}+z^{\ell j}\right )\\
	&=\sum_{j=1}^\infty \frac{ h_{j-1}\left (z^{2 j}+(k-1) z^{3 j}\right ) }{\left(z;z\right)_j^2} z^{\ell j} \left(1-z^j\right )^2\\
	&=\sum_{j=1}^\infty \frac{ h_{j-1}\left (z^{(\ell+2) j}+(k-1) z^{(\ell+3) j}\right ) }{\left(z;z\right)_{j-1}^2}\\
	&=\sum_{j=0}^\infty \frac{ h_{j}\left (z^{(\ell+2) (j+1)}+(k-1) z^{(\ell+3) (j+1)}\right ) }{\left(z;z\right)_{j}^2}\\
	&=z^{\ell+2} A_{\ell+2}+ (k-1)z^{\ell+3} A_{\ell+3}.
	\end{align*}
	
	Similarly, we can prove that $B_\ell$ is a solution. The interested reader can find the detailed calculation for this in the Appendix.
\end{proof}

We have yet to find a solution that satisfies the boundary conditions. A suitable linear combination of $A_\ell$ and $B_\ell$, however, does the trick. In general, setting $F_\ell = \alpha A_\ell + \beta B_\ell$ and solving the simultaneous equations $c_1 F_1 + c_2 F_2 + c_3 F_3 = 1$ and $d_1 F_1 + d_2 F_2 + d_3 F_3 = 1$ for $\alpha$ and $\beta$ yields after some algebra:
\begin{align*}
\alpha &= \big(-(c_1-d_1)B_1-(c_2-d_2)B_2-(c_3-d_3)B_3\big)/d,\\
\beta &= \big((c_1-d_1)A_1+(c_2-d_2)A_2+(c_3-d_3)A_3\big)/d,\text{ where}\\
d &= (c_1 A_1 + c_2 A_2 + c_3 A_3)(d_1 B_1 + d_2 B_2 + d_3 B_3)\\
&\quad-(c_1 B_1 + c_2 B_2 + c_3 B_3)(d_1 A_1 + d_2 A_2 + d_3 A_3).
\end{align*}
Now setting
\begin{align*}
c_1 &= 1-(2k-1)z,&c_2&=-(k-1)z^2,&c_3&=0,\\
d_1 &=-(3k-1)z,&d_2&=1-(2k-1)z^2,&d_3&=-(k-1)z^3
\end{align*} as in \eqref{boundary} and plugging $\alpha$ and $\beta$ into $F_\ell = \alpha A_\ell + \beta B_\ell$ yields the result of Theorem~\ref{thm:rowconvex}.

\section{Comments and open questions}

\begin{enumerate}
	\item One might reconsider the \emph{restricted} problem. Using a substitution similar to the one mentioned in Corollary~\ref{domino:4^n}, we can deduce the generating function for the restricted case, by replacing $z y_i$ with $\frac{z y_i}{1+z y_i}$.
	However, is it possible to find a direct way of enumerating the generating function of restricted towers and a closed formula for its coefficients?
	\item In this paper we have counted $S$-omino towers and row-convex towers. Is it possible to combine the two ideas and count row-convex $S$-omino towers?
\end{enumerate}

\section{Acknowledgement}
I would like to thank my supervisor Anusch Taraz and Julian Gro\ss{}mann for many helpful comments and suggestions regarding both content and presentation.

\section*{Appendix}
For the sake of completeness, we show that $B_\ell$ is a solution to \eqref{recF}.
\begin{align*}
&B_{\ell+2}-2 B_{\ell+1} + B_\ell\\
=&\sum_{j=0}^\infty \frac{ h_j }{\left(z;z\right)_j^2} \Bigg ((\ell+2)z^{(\ell+2) j}-2(\ell+1)z^{(\ell+1) j}+\ell z^{\ell j} \\
&\quad\quad\quad\quad\quad\quad+\left(z^{(\ell+2) j}-2z^{(\ell+1) j}+z^{\ell j}\right)\sum _{m=1}^j \left (1+\frac{2}{1-z^m}-\frac{1}{1+(k-1) z^m} \right)\Bigg )\\
=&\sum_{j=1}^\infty \frac{ h_j }{\left(z;z\right)_j^2} z^{\ell j} \left(1-z^j\right )^2 \Bigg (2+\ell-\frac{2}{1-z^j} + \sum _{m=1}^j \left (1+\frac{2}{1-z^m}-\frac{1}{1+(k-1) z^m} \right)\Bigg )\\
=&\sum_{j=1}^\infty \frac{ h_{j-1}\left (z^{(\ell+2) j}+(k-1) z^{(\ell+3) j}\right ) }{\left(z;z\right)_{j-1}^2} \Bigg (2+\ell-\frac{2}{1-z^j}\\
& \quad\quad\quad\quad\quad\quad\quad\quad\quad\quad\quad\quad\quad\quad\quad\quad+ \sum _{m=1}^j \left (1+\frac{2}{1-z^m}-\frac{1}{1+(k-1) z^m} \right)\Bigg )\\
=&\sum_{j=1}^\infty \frac{ h_{j-1}\left (z^{(\ell+2)j}+(k-1) z^{(\ell+3)j}\right ) }{\left(z;z\right)_{j-1}^2} \Bigg (3+\ell-\frac{1}{1+(k-1)z^{j}} \\
&\quad\quad\quad\quad\quad\quad\quad\quad\quad\quad\quad\quad\quad\quad\quad\quad+\! \sum _{m=1}^{j-1} \! \! \left (1+\frac{2}{1-z^m}-\frac{1}{1+(k-1) z^m} \right)\! \Bigg )\\
=&\sum _{j=1}^{\infty} \frac{ h_{j-1} }{\left(z;z\right)_{j-1}^2} \Bigg ((\ell+2) z^{(\ell+2) j}+(k-1)(\ell+3) z^{(\ell+3) j}\\
&+ \left (z^{(\ell+2)j}+(k-1)z^{(\ell+3)j}\right)\sum _{m=1}^{j-1} \left (1+\frac{2}{1-z^m}-\frac{1}{1+(k-1) z^m} \right) \Bigg )\\
=&\sum _{j=0}^{\infty} \frac{ h_j }{\left(z;z\right)_j^2} \Bigg ((\ell+2) z^{(\ell+2) (j+1)}+(k-1)(\ell+3) z^{(\ell+3) (j+1)}\\
&+ \left (z^{(\ell+2)(j+1)}+(k-1)z^{(\ell+3)(j+1)}\right )\sum _{m=1}^j \left (1+\frac{2}{1-z^m}-\frac{1}{1+(k-1) z^m} \right) \Bigg )\\
=&z^{\ell+2} B_{\ell+2}+ (k-1)z^{\ell+3} B_{\ell+3}.
\end{align*}


\begin{thebibliography}{1}
	
	\bibitem{miklosBonaHandbook}
	M.\ B{\'o}na, editor.
	\newblock {\em Handbook of Enumerative Combinatorics}.
	\newblock CRC Press, 2015.
	
	\bibitem{1608.01562}
	T.\ M.\ Brown.
	\newblock Convex domino towers.
	\newblock {\em J.of Integer Sequences} \textbf{20} (2017).
	
	\bibitem{1608.01563}
	T.\ M.\ Brown.
	\newblock On the enumeration of {$k$}-omino towers.
	\newblock {\em Discrete Mathematics} \textbf{340} (2017), 1319--1326.
	
	\bibitem{flajolet}
	P.\ Flajolet and R.\ Sedgewick.
	\newblock {\em Analytic Combinatorics}.
	\newblock Cambridge University Press, 2009.
	
	\bibitem{1609.05988}
	I.\ M.\ Gessel.
	\newblock Lagrange inversion.
	\newblock {\em J.\ of Combinatorial Theory Series A} \textbf{144} (2016), 212--249.
	
	\bibitem{PhysRevLett.60.1107}
	V.~Privman and N.~M.\ {\v{S}}vraki\'{c}.
	\newblock Exact generating function for fully directed compact lattice animals.
	\newblock {\em Physical Review Letters} \textbf{60} (1988), 1107--1109.
	
	\bibitem{Rukavicka2011}
	J.\ Rukavicka.
	\newblock On generalized {D}yck paths.
	\newblock {\em Electronic J.\ of Combinatorics} \textbf{18} (2011).
	
	\bibitem{MR768962}
	G.\ Viennot.
	\newblock Probl\`emes combinatoires pos\'{e}s par la physique statistique.
	\newblock {\em Ast\'{e}risque} \textbf{121-122} (1985), 225--246.
	
	\bibitem{1208.2258}
	D.\ Zeilberger.
	\newblock The amazing $3^n$ theorem and its even more amazing proof [Discovered
	by Xavier G.\ Viennot and his École Bordelaise gang].
	\newblock {\em Personal J.\ of Shalosh B.\ Ekhad and Doron Zeilberger}, 2012.
	
\end{thebibliography}
\end{document}